\documentclass{article}

\usepackage{amsmath}
\usepackage{amssymb}
\usepackage{amsthm}
\usepackage{mathrsfs}
\usepackage{amscd}
\usepackage[all,cmtip]{xy}
\usepackage{enumerate}
\usepackage{comment}

\usepackage[alphabetic,backrefs,lite]{amsrefs}

\newcommand{\F}{{\mathbb F}}

\newcommand{\PP}{{\mathbb P}}
\newcommand{\Q}{{\mathbb Q}}

\newcommand{\Z}{{\mathbb Z}}

\DeclareMathOperator{\rank}{rank}

\newcommand{\GL}{\operatorname{GL}}

\newcommand{\tors}{{\operatorname{tors}}}

\newtheorem{theorem}{Theorem}[section]
\newtheorem{proposition}[theorem]{Proposition}
\newtheorem{lemma}[theorem]{Lemma}
\newtheorem{corollary}[theorem]{Corollary}

\title{Torsion points of elliptic curves over multi-quadratic number fields}
\author{Koji Matsuda\thanks{University of Tokyo}}
\date{}

\begin{document}

\maketitle

\begin{abstract}
We compute the Mordell-Weil groups of the modular Jacobian varieties of hyperelliptic modular curves $X_1(M, MN)$
over every number field which is the composition of quadratic fields.
Also we prove criteria for the existence of elliptic curves over such number fields with prescribed torsion points
generalizing the results for quadratic number fields of Kamienny and Najman.
\end{abstract}

\section{Notation}

For an abelian group $A$ and a nonzero integer $n$, we denote $A[n^\infty] = \cup_{i \ge 1} A[n^i]$.
For positive integers $d_1, \dots, d_n$, we denote $\bigoplus_{i=1}^n \Z/d_i$ by $[d_1, \dots, d_n].$
For schemes $X/S$ and $T/S$, we denote the base change of $X$ with respect to $T/S$ by $X_T$.

\section{Introduction}

For an elliptic curve $E$ over a number field $K$, the Mordell-Weil group $E(K)$ is finitely generated,
in particular its torsion group is finite.
When $K = \Q$, Mazur classified which groups appear as the torsion part of that:

\begin{theorem}[\cite{Mazur}]
Let $E/\Q$ be an elliptic curve.
Then the torsion part $E(\Q)_\tors$ of its Mordell-Weil group is isomorphic to one of the following 15 groups:
\begin{align*}
& \Z/N\Z & \text{ for } & 1 \leq N \leq 12, N \neq 11, & \\
& \Z/2\Z \times \Z/2N\Z & \text{ for } & 1 \leq N \leq 4. &
\end{align*}
\end{theorem}

Similarly, the classification for quadratic number fields was started by Kenku and Momose,
and completed by Kamienny:

\begin{theorem}[\cite{KeMo}, \cite{Kamienny}]
Let $K/\Q$ be a quadratic number field and $E/K$ an elliptic curve.
Then the torsion part $E(K)_\tors$ of its Mordell-Weil group is isomorphic to one of the following 26 groups:
\begin{align*}
& \Z/N\Z & \text{ for } & 1 \leq N \leq 18, N \neq 17, & \\
& \Z/2\Z \times \Z/2N\Z & \text{ for } & 1 \leq N \leq 6, & \\
& \Z/3\Z \times \Z/3N\Z & \text{ for } & 1 \leq N \leq 2, & \\
& \Z/4\Z \times \Z/4\Z. &
\end{align*}
\end{theorem}

If we fix a quadratic field $K$, then not all groups from this theorem occur as the torsion groups of elliptic curves over $K$.
For example if $K$ does not contain a primitive $m$-th root of unity, then $\Z/mZ \times \Z /m\Z$ cannot
appear as the torsion groups over $K$.

In \cite{Najman}, Najman completely classified all the possible torsion groups over each of $\Q(\sqrt{-1})$ and $\Q(\sqrt{-3})$,
and in \cite{KaNa} Kamienny and Najman generalized it to obtain methods to determine all the possible torsion groups over
a fixed quadratic number field $K$.

In this paper, we generalize the methods of \cite{KaNa} and obtain methods to determine,
for a finite abelian group $\Z/n \times \Z/nm$ where $(n,nm)$ satisfies that $X_1(n,nm)$ is a hyperelliptic curve,
whether it appears as the torsion groups over a fixed multi-quadratic number field,
i.e., a number field (including $\Q$) which is a composition of quadratic number fields.
For example, we show that, for a multi-quadratic number field $K$, if the rank of $J_1(13)(K)$ is zero,
then there are no elliptic curves $E$ over $K$ whose Mordell-Weil group contains $13$-torsion points.

For explicit computation of, for example the Mordell-Weil groups of the Jacobian varieties of hyperelliptic curves,
we use Magma \cite{Magma}.

\section{Preliminaries}

For a positive integer $N$ and a subgroup $H$ of $\GL_2(\Z/N \Z)$,
we denote the algebraic stack over $\Z[1/N]$ classifying generalized elliptic curves with the level $H$ structure by $\mathscr{M}_H(N)$,
and denote its coarse moduli space by $X_H(N)$.
This is actually a scheme, is smooth, and has geometrically connected fibers over $\Z[1/N][\zeta_N]^{\det H}$.
If for every object $x \in \mathscr{M}_H(N)(k)$ over any algebraically closed field $k$ of characteristic prime to $N$, the group of automorphisms of $x$ is trivial, then the stack $\mathscr{M}_H(N)[1/N]$ is a scheme.

For particular subgroups $H = 1, \Gamma_1(N), \Gamma_1(M, MN)$ and $\Gamma_0(N)$
(where
\begin{eqnarray*}
\Gamma_1(M,MN) = 
\left\{ \left( \begin{matrix}1 & b \\ 0 & d \end{matrix} \right) \in \GL_2 (\Z/MN) \bigg |
b \equiv 0, d \equiv 1 \mod M \right\})
\end{eqnarray*}
respectively, we use the notation
$\mathscr{M}(N), \mathscr{M}_1(N), \mathscr{M}_1(M, MN)$ and $\mathscr{M}_0(N)$ respectively
and similarly for other objects associated to them.

There are many algorithms computing explicit equations of modular curves.
We use models of $X_1(N)$ given in \cite{Baaziz},
and for other modular curves \\
$X_1(M,MN)$, we use models given in \cite{Sutherland},
which computed them using the method of \cite{DeSu}.

For computations of the Mordell-Weil groups of the Jacobian varieties of hyperelliptic curves over a number field,
it is convenient to use the following lemma (for example, see \cite[1.3]{LaLo}):

\begin{proposition} \label{exp2twists}
Let $k$ be a number field, $K/k$ a finite Abelian extension of exponent $2$, 
$C$ a hyperelliptic curve over $k$, and $J$ its Jacobian variety.
Then $\operatorname{rank}J(K) = \sum_d \operatorname{rank}J^{(d)}(k),$ and
for any odd integer $m$, we have $J(K)[m] = \oplus_d J^{(d)}(k)[m],$
where the sum is taken over all quadratic twist of $C$ over $k$ trivialized by $K$ and $J^{(d)}$ its Jacobian.
\end{proposition}

\section{The torsion parts of the modular Jacobian varieties} \label{Jacobian}

In this section, $K$ is a multi-quadratic number field unless otherwise stated.
First we state a trivial but crucial lemma:

\begin{lemma}
For any prime $\mathfrak{p}$ of $K$ whose characteristic of its residue field is odd,
the inertia degree and the ramification index are $1$ or $2$.
\end{lemma}

\begin{proof}
The quotient of the decomposition group by the inertia group
and the quotient of the inertia group by the ramification group are cyclic.
\end{proof}

Using it and Proposition \ref{exp2twists} we can compute the Mordell-Weil group $J_1(N)$ for hyperelliptic modular curves easily.
In the proofs of theorems in this section we write for simplicity $X$ and $J$ the modular curve and its Jacobian treated in the theorems.

\begin{theorem} \label{Jacobian11}
$J_1(11)(K)_{\tors} =J_1(11)(\Q) \cong \Z/5$.
\end{theorem}

\begin{proof}
The modular curve $X$ has a model $y^2 - y = x^3 - x^2.$
Thus we can compute $J(\Q) = \Z/5$.
Next we can compute $J(\F_{3^2}) = \Z/15$ and $J(\F_{5^2}) = \Z/35$ (for example use Magma).
Thus the result.
\end{proof}

\begin{theorem} \label{Jacobian13}
$J_1(13)(K)_{\operatorname{tor}} =J_1(13)(\Q) \cong \Z/19$.
\end{theorem}

\begin{proof}
The modular curve $X$ has a model $y^2 = x^6 - 2 x^5 + x^4 - 2 x^3 + 6 x^2 - 4 x + 1 $.
Using it we can compute $J(\Q) = \Z/19, J(\F_9) = \Z/3 \times \Z/19,$
and $J(\F_{25}) = \Z/19 \times \Z/19$.
Hence the result.
\end{proof}

Henceforth for convenience we use the following notation:
For a number field $K$ and a positive integer $N$,
define $K_{(N)} = K \cap \Q(\zeta_N)$.

\begin{theorem} \label{Jacobian14}
\begin{align*}
J_1(14)(K)_{\tors} & = J_1(14)(K \cap \Q(\zeta_{14})) \\
 & \cong \left\{
\begin{array}{ll}
\Z/6 & (K_{(14)} = \Q) \\
\Z/2 \times \Z/6 & (K_{(14)} = \Q(\sqrt{-7})).
\end{array}
\right.
\end{align*}
\end{theorem}

\begin{proof}
Using a model $y^2 = f(x)$ for $f(x) = (x+33)(x^2 - 33x + 414)$ of $X$, we can compute
$J(\Q) = \Z/6, J(\F_9) = \Z/2 \times \Z/6,$ and
$J(\F_{13^2}) = \Z/2 \times \Z/90$.

By Proposition \ref{exp2twists}, $J(K)[3^\infty] = \oplus_d J^{(d)}(\Q)[3^\infty]$,
where the direct sum is taken over all quadratic twists $J^{(d)}$.
The left hand side is contained in $J(\F_{13^2})[3^\infty] = \Z/9$, and
the right hand side contains $J(\Q)[3^\infty] = \Z/3$ as its direct summand.
Therefore $J(K)[3^\infty]$ must be $\Z/3$.
On the other hand $J(K)$ has full $2$ torsion points if and only if $f(x)$ splits in $K$, i.e., $\sqrt{-7} \in K$.

Finally we can compute that the ranks of the Mordell-Weil group of $J$ over $\Q$ and over $\Q(\sqrt{-7})$ are zero
using $2$-descent.
Therefore the result.
\end{proof}

\begin{theorem} \label{Jacobian15}
\begin{align*}
J_1(15)(K)_{\tors} & = J_1(15)(K \cap \Q(\zeta_{15})) \\
 & \cong \left\{
\begin{array}{ll}
\Z/4 & (K_{(15)} = \Q) \\
\Z/8 & (K_{(15)} = \Q(\sqrt{-3}) \text{ or } \Q(\sqrt{5}) ) \\
\Z/2 \times \Z/4 & (K_{(15)} = \Q(\sqrt{-15})) \\
\Z/2 \times \Z/8 & (K_{(15)} = \Q(\sqrt{-3}, \sqrt{5})).
\end{array}
\right.
\end{align*}
\end{theorem}

\begin{proof}
Using a model $y^2 = f(x)$ for $f(x) = (x+21)(x^2 - 21x + 414)$ of $X$ we can compute
$J(\Q) = \Z/4, J(\F_{7^2}) = \Z/8 \times \Z/8,$ and
$J(\F_{13^2}) = \Z/2 \times \Z/96$.
Thus $\Z/4 \subseteq J(K)_{\tors} \subseteq  \Z/2 \times \Z/8.$

We compute when $J(K)$ contains $\Z/2 \times \Z/2$ or $\Z/8$.
First, $J(K)$ contains $\Z/2 \times \Z/2$ if and only if $f(x)$ splits in $K$, i.e., $\sqrt{-15} \in K$.
Next, for the $n$-th division polynomials $\varphi_n$,
$J(K)$ has a $8$-torsion point if and only if the polynomial $\varphi_8 / \varphi_4$ has a root $a$ in $K$
and $K$ contains a root of $y^2 - f(a)$.
According to Magma, the irreducible factors of $\varphi_8 / \varphi_4$ with the Galois group of exponent $2$
are $x^2 - 66 x - 531$ and $x^2 + 6x + 981.$
By an easy computation, we see that, for a root $a$ of one of these two polynomials,
the number field $\Q(\sqrt{a})$ contains roots of $y^2 - f(a)$.
Hence $J(K)$ has a $8$-torsion point if and only if at least one of them has roots in $K$,
i.e., $\sqrt{-3}$ or $\sqrt{5} \in K$.

Finally by $2$ descent we can compute that the rank of $J$ over $\Q(\sqrt{-3}, \sqrt{5})$ is zero.
Summarizing them we obtain the result.
\end{proof}

\begin{theorem} \label{Jacobian16}
\begin{align*}
J_1(16)(K)_{\tors} & = J_1(16)(K \cap \Q(\zeta_{16})) \\
 & \cong \left\{
\begin{array}{ll}
\Z/2 \times \Z/10 & (K_{(16)} = \Q \text{ or } \Q(\sqrt{-2})) \\
\Z/2 \times \Z/2 \times \Z/10 & (K_{(16)} = \Q(\sqrt{-1}) \text{ or } \Q(\sqrt{2})) \\
\Z/2 \times \Z/2 \times \Z/2 \times \Z/10 & (K_{(16)} = \Q(\sqrt{-1}, \sqrt{2})).
\end{array}
\right.
\end{align*}
\end{theorem}

\begin{proof}
Using a model $y^2 = x(x^2 + 1)(x^2 + 2x - 1)$ of $X$, we can compute
$J(\Q) = [2,10], J(\F_9) = [2,2,2,10],$ and $J(\F_{5^2}) = [2,2,4,40].$
Thus $J(K) \subseteq [2,2,2,10].$

Now considering the Weierstrass points of the hyperelliptic model of $X$, we can compute $J(K)[2]$.
Finally computing the ranks using $2$ descent, we have the results.
\end{proof}

\begin{theorem} \label{Jacobian18}
\begin{align*}
J_1(18)(K)_{\tors} & = J_1(18)(K \cap \Q(\zeta_{18})) \\
 & \cong \left\{
\begin{array}{ll}
\Z/21 & (K_{(18)} = \Q) \\
\Z/3 \times \Z/21 & (K_{(18)} = \Q(\sqrt{-3})).
\end{array}
\right.
\end{align*}
\end{theorem}

\begin{proof}
Using a model $y^2 = x^6 + 2x^5 + 5x^4 + 10x^3 + 10x^2 + 4x + 1$ of $X$, we can compute
$J(\Q) = \Z/21, J(\F_{7^2}) = [3, 651]$ and $J(\F_{11^2}) = [12, 1092]$.
Thus $J(K) \subseteq [3,21]$.
Now according to Magma, the twist of $J$ over $\Q(\sqrt{-3})$ has the Mordell-Weil group $\Z/3$
over $\Q$.

Thus computing the ranks we consider using $2$ descent we have the result.
\end{proof}

\begin{theorem} \label{Jacobian2,10}
\begin{align*}
J_1(2,10)(K)_{\tors} & = J_1(2,10)(K \cap \Q(\zeta_{10})) \\
 & \cong \left\{
\begin{array}{ll}
\Z/6 & (K_{(10)} = \Q) \\
\Z/2 \times \Z/6 & (K_{(10)} = \Q(\sqrt{5})).
\end{array}
\right.
\end{align*}
\end{theorem}

\begin{proof}
The modular curve $X$ is the elliptic curve with Cremona label 20a2.
We compute that $J(\Q) = \Z/6, J(\F_9) = [2,6],$ and $J(\F_{7^2}) = [2,30].$
Then by a similar calculation to the proof of \ref{Jacobian11} we have the result.
\end{proof}

\begin{theorem} \label{Jacobian2,12}
\begin{align*}
J_1(2,12)(K)_{\tors} & = J_1(2,12)(K \cap \Q(\zeta_{12})) \\
& \cong \left\{
\begin{array}{ll}
\Z/4 & (K_{(12)} = \Q) \\
\Z/8 & (K_{(12)} = \Q(\sqrt{-1}) \text{ or } \Q(\sqrt{3})) \\
\Z/2 \times \Z/4 & (K_{(12)} = \Q(\sqrt{-3})) \\
\Z/2 \times \Z/8 & (K_{(12)} = \Q(\sqrt{-1}, \sqrt{3})).
\end{array}
\right.
\end{align*}
\end{theorem}

\begin{proof}
The modular curve $X$ is the elliptic curve with Cremona label 24a4.
We can compute $J(\Q) = \Z/4,
J(\F_{5^2}) = [2,16], J(\F_{7^2}) = [8,8]$ and the division polynomial $\varphi_8/\varphi_4$.
Therefore by a similar calculation to the proof of \ref{Jacobian11} we obtain the result.
\end{proof}

The modular curve $X_1(M,MN)$ is defined over $\Q(\zeta_M)$.
So henceforth in the theorem treating $X_1(M,MN)$, we assume that $\Q(\zeta_M) \subseteq K$.

\begin{theorem} \label{Jacobian3,9}
$J_1(3,9)(K)_{\tors} = J_1(3,9)(\Q(\sqrt{-3})) \cong \Z/3 \times \Z/3$
\end{theorem}

\begin{proof}
The modular curve $X$ is defined over $\Q(\sqrt{-3})$
and it is the base change of the elliptic curve with Cremona label 27a3.
We compute that $J(\Q(\sqrt{-3})) = [3,3], J(\F_{25}) = [6,6],$
and $J(\F_{7^2}) = [3,21]$.
Thus for all $\Q(\sqrt{-3}) \subseteq K$, we have $J(K)_\tors = J(\Q(\sqrt{-3})).$
\end{proof}

\begin{theorem} \label{Jacobian4,8}
\begin{align*}
J_1(4,8)(K)_{\tors} & = J(K \cap \Q(\zeta_8)) \\
& \cong \left\{
\begin{array}{ll}
\Z/2 \times \Z/4 & ( K_{(8)} = \Q(\sqrt{-1}) ) \\
\Z/4 \times \Z/4 & ( K_{(8)} = \Q(\sqrt{-1}, \sqrt{2}) ).
\end{array}
\right.
\end{align*}
\end{theorem}

\begin{proof}
The modular curve $X$ is defined over $\Q(\sqrt{-1})$
and it is the base change of the elliptic curve with Cremona label 32a2.
We compute that $J(\Q(\sqrt{-1})) = [2,4], J(\F_9) = [4,4],$
and $J(\F_{25}) = [4,8]$.
Thus for all $\Q(\sqrt{-1}) \subseteq K$, we have $[2,4] \subseteq J(K)_\tors \subseteq [4,4]$.
Finally computing the second and forth division polynomials, we obtain the result.
\end{proof}

\begin{theorem} \label{Jacobian6,6}
$J_1(6,6)(K)_{\tors} = J_1(6,6)(\Q(\sqrt{-3})) \cong \Z/2 \times \Z/6$.
\end{theorem}

\begin{proof}
The modular curve $X$ is defined over $\Q(\sqrt{-3})$
and it is the base change of the elliptic curve with Cremona label 36a1.
We compute that $J(\Q(\sqrt{-3})) = [2,6], J(\F_{25}) = [6,6],$
and $J(\F_{49}) = [4,12]$.
Hence the result.
\end{proof}

\section{The rational points on the modular curves} \label{rationals}

Using the description of the Mordell-Weil groups of the modular Jacobian varieties above,
we deduce the statements about torsion points of elliptic curves over multi-quadratic number fields easily.
For the case of quadratic fields, these are already known. (See \cite{KaNa}.)

In the following proofs, we use the fact that for $N \ge 4$ the moduli stack $\mathscr{M}_1(N)[1/N]$ is a scheme.
All statements about elliptic curves over a field $K$ is considered up to $K$-isomorphism.

\begin{theorem} \label{11torsion}
The following are equivalent:
\begin{enumerate}[(1)]
\item There exists an elliptic curve over $K$ whose Mordell-Weil group contains $\Z/11$. \label{11torsion1}
\item There exist infinitely many such curves. \label{11torsionfinite}
\item $\rank J_1(11)(K) \ge 1.$ \label{11rank}
\end{enumerate}
\end{theorem}

As we have noted, using twists we can check the rank condition by only checking the Mordell-Weil groups
of $\deg K$ Jacobian varieties of hyperelliptic curves over $\Q$.
It means that the calculation is very fast.

\begin{proof}
By Theorem \ref{Jacobian11} we have that the condition $\rank J(K) = 0$ is equivalent to
$X(K) = X(\Q)$.
Since we know that all points of $X(\Q)$ are cusps the result follows.
\end{proof}

For genus $\ge 2$ cases, we can get only a partial connection between the existence of certain torsions on elliptic curves
and the rank of the modular Jacobian variety corresponding to the torsion we consider.
However, we can obtain more properties about such an existence.
See \cite[section 2.6.]{Krumm}.

\begin{theorem} \label{13torsion}
Consider the following conditions:
\begin{enumerate}[(1)]
\item There exists an elliptic curve over $K$ whose Mordell-Weil group contains $\Z/13$. \label{13torsion1}
\item There exist finitely many such curves. \label{13torsionfinite}
\item $\rank J_1(13)(K) \ge 1.$ \label{13rank}
\end{enumerate}
Then we have (\ref{13torsion1}) $\Leftrightarrow$ (\ref{13torsionfinite}) $\Rightarrow$ (\ref{13rank}).
\end{theorem}

\begin{proof}
The first equivalence is just Faltings' theorem.
Assume that the rank is $0$. Then by Theorem \ref{Jacobian13},
we have that $X(K) \subseteq J(K)_{\tors} = J(\Q).$
Therefore every points of $X(K)$ are fixed by the Galois group of $K/\Q$.
Hence, since every $\Q$-rational point of $X$ is cusp, we have (\ref{13torsion1}) $\Rightarrow$ (\ref{13rank}).
\end{proof}

\begin{theorem} \label{14torsion}
Consider the following conditions:
\begin{enumerate}[(1)]
\item There exists an elliptic curve over $K$ whose Mordell-Weil group contains $\Z/14$. \label{14torsion1}
\item There exist at least 3 such curves. \label{14torsion3}
\item There exist infinitely many such curves. \label{14torsioninfinite}
\item $\rank J_1(14)(K) \ge 1.$ \label{14rank}
\end{enumerate}
If $\sqrt{-7} \not\in K$ then (\ref{14torsion1}) $\Leftrightarrow$ (\ref{14torsioninfinite}) $\Leftrightarrow$ (\ref{14rank}).
If $\sqrt{-7} \in K$ then there always exist at least 2 such curves and
(\ref{14torsion3}) $\Leftrightarrow$ (\ref{14torsioninfinite}) $\Leftrightarrow$ (\ref{14rank}).
\end{theorem}

\begin{proof}
The first case is exactly the same as the Theorem \ref{11torsion}.
For the second case, by Theorem \ref{Jacobian14} we have $\# X(K) = 12$ or $= \infty$.
Thus counting the number of cusps, if the rank of $J(K)$ is zero, then we have that $\# Y(K) = 6$.
Now over a quadratic field, any elliptic curves have no subgroup $\Z/2 \times \Z/14$ in its Mordell-Weil groups.
(\cite{Kamienny})
Therefore, using $\# ( (\Z/14)^* / \{ \pm 1 \} ) = 3,$ we have that over $\Q(\sqrt{-7})$ there
exist exactly 2 elliptic curves whose Mordell-Weil groups contain $\Z/14$.
Hence the result.
\end{proof}

From this theorem we obtain the following corollary stating that if an elliptic curve over a large number field
has a certain torsion point then it is actually defined over a much smaller number field:

\begin{corollary} \label{exceptional14}
The only elliptic curves over $\Q(\sqrt{-7})$ whose Mordell-Weil groups contain $\Z/14$ are:
\begin{align*}
y^2 + \frac{14 + 3 \sqrt{-7}}{7} xy + \frac{-3 + \sqrt{-7}}{7} y = x^3 +  \frac{-3 + \sqrt{-7}}{7} x^2 \\
y^2 + \frac{7 + 2 \sqrt{-7}}{7} xy + \frac{1 + \sqrt{-7}}{7} y = x^3 +  \frac{1 + \sqrt{-7}}{7} x^2
\end{align*}

Further for a multi-quadratic number field $K$ containing $\sqrt{-7}$,
if the rank of $J_1(14)(K)$ is zero, then every elliptic curve over $K$ which contains $\Z/14$ in its Mordell-Weil group
is defined over $\Q(\sqrt{-7})$.
\end{corollary}

The equations are taken from \cite[Theorem 16.]{KaNa}.
These can be computed using the explicit equations of non cuspidal points for a model of $X_1(14)$
and its universal elliptic curve.

\begin{theorem} \label{15torsion}
Consider the following conditions:
\begin{enumerate}[(1)]
\item There exists an elliptic curve over $K$ whose Mordell-Weil group contains $\Z/15$. \label{15torsion1}
\item There exist at least 2 such curves. \label{15torsion2}
\item There exist at least 3 such curves. \label{15torsion3}
\item There exist infinitely many such curves. \label{15torsioninfinite}
\item $\rank J_1(15)(K) \ge 1.$ \label{15rank}
\end{enumerate}
If neither of $\sqrt{5}$ nor $\sqrt{-15}$ is in $K$ then (\ref{15torsion1}) $\Leftrightarrow$ (\ref{15torsioninfinite}) $\Leftrightarrow$ (\ref{15rank}).
If only one of $\sqrt{5}$ or $\sqrt{-15}$ is in $K$ then there always exists at least 1 such curve and
(\ref{15torsion2}) $\Leftrightarrow$ (\ref{15torsioninfinite}) $\Leftrightarrow$ (\ref{15rank}).
If both of $\sqrt{5}$ and $\sqrt{-15}$ are in $K$ then there always exist at least 2 such curves and
(\ref{15torsion3}) $\Leftrightarrow$ (\ref{15torsioninfinite}) $\Leftrightarrow$ (\ref{15rank}).
\end{theorem}

\begin{proof}
Exactly the same as Theorem \ref{14torsion}.
\end{proof}

Similarly we have the following corollary:

\begin{corollary} \label{exceptional15}
The only elliptic curve over $\Q(\sqrt{5})$ whose Mordell-Weil group contains $\Z/15$ is:
\begin{align*}
y^2 = x^3 +  (281880 \sqrt{5} - 630315) x + 328392630 - 146861640 \sqrt{5}.
\end{align*}

The only elliptic curve over $\Q(\sqrt{-15})$ whose Mordell-Weil group contains $\Z/15$ is:
\begin{align*}
y^2 + \frac{145 + 7 \sqrt{-15}}{128} xy + \frac{265 + 79 \sqrt{-15}}{4096} y = x^3 +  \frac{265 + 79 \sqrt{-15}}{4096} x^2.
\end{align*}

Further for a multi-quadratic number field $K$ containing $\sqrt{5}$ or $\sqrt{-15}$,
if $\rank J_1(15)(K) = 0$, then every elliptic curve over $K$ which contains $\Z/15$ in its Mordell-Weil group
is defined over $\Q(\sqrt{5})$ or $\Q(\sqrt{-15})$.
\end{corollary}

The equations are also taken from \cite{KaNa}.

The curves in Corollary \ref{exceptional14} are not defined over $\Q$,
but the curves in Corollary \ref{exceptional15} are defined over $\Q$.
See \cite[Theorem 2]{Najman2}.

\begin{theorem} \label{16torsion}
Consider the following conditions:
\begin{enumerate}[(1)]
\item There exists an elliptic curve over $K$ whose Mordell-Weil group contains $\Z/16$. \label{16torsion1}
\item There exist finitely many such curves. \label{16torsionfinite}
\item $\rank J_1(16)(K) \ge 1.$ \label{16rank}
\end{enumerate}
Then we have (\ref{16torsion1}) $\Leftrightarrow$ (\ref{16torsionfinite}) $\Rightarrow$ (\ref{16rank}).
\end{theorem}

\begin{proof}
The first equivalence is just Faltings' theorem.
Assume that the rank is $0$. Then by Theorem \ref{Jacobian16},
we have that $X(K) \subseteq J(K)_{\tors}$ and the later group is equal to $J(K \cap \Q(\sqrt{-1}, \sqrt{2}))$.
Therefore $X(K) = X(K \cap \Q(\sqrt{-1}, \sqrt{2})).$
So in order to show $Y(K) = \varnothing$, it suffices to show that all points of $X(\Q(\sqrt{-1}, \sqrt{2}))$ are cusps.

Over $\F_9$, by the Hasse bound an elliptic curve has order at most $16$ points.
Computing all elliptic curves (using Magma), we have that
no elliptic curves contain subgroups isomorphic to $\Z/16$ in the group of the rational points.
Therefore we have that $X(\F_9) = \varnothing$.
On the other hand since we are assuming that the rank is $0$, under the reduction map at a prime above $3$,
the prime to $3$ part of $J(K)$ injects into $J(\F_9)$.
By Theorem \ref{Jacobian16}, it is the whole of $J(K)$, hence $J(K)$ injects into $J(\F_9)$.
Therefore since of course $X(K) \to J(K)$ is injective, we have that $X(K) \to X(\F_9)$ is also injective.
Thus counting cusps over $\Q(\sqrt{-1}, \sqrt{2})$ and over $\F_9$, we have that all points of $X(K)$ are cusps.
\end{proof}

\begin{theorem} \label{18torsion}
Consider the following conditions:
\begin{enumerate}[(1)]
\item There exists an elliptic curve over $K$ whose Mordell-Weil group contains $\Z/18$. \label{18torsion1}
\item There exist finitely many such curves. \label{18torsionfinite}
\item $\rank J_1(18)(K) \ge 1.$ \label{18rank}
\end{enumerate}
Then we have (\ref{18torsion1}) $\Leftrightarrow$ (\ref{18torsionfinite}) $\Rightarrow$ (\ref{18rank}).
\end{theorem}

\begin{proof}
Exactly the same as the theorems above.
The key point is that $Y(\Q(\sqrt{-3})) = \varnothing$.
This is \cite{Najman}.
\end{proof}

Henceforth as in the section \ref{Jacobian}, we assume that $\Q(\zeta_M) \subseteq K$ in the
theorem treating $X_1(M,MN)$.
All proofs are completely the same as the proof of Theorem \ref{11torsion}.

\begin{theorem} \label{2,10torsion}
The following are equivalent:
\begin{enumerate}[(1)]
\item There exists an elliptic curve over $K$ whose Mordell-Weil group contains
	$\Z/2 \times \Z/10$. \label{2,10torsion1}
\item There exist infinitely many such curves. \label{2,10torsionfinite}
\item $\rank J_1(2,10)(K) \ge 1.$ \label{2,10rank}
\end{enumerate}
\end{theorem}

\begin{theorem} \label{2,12torsion}
The following are equivalent:
\begin{enumerate}[(1)]
\item There exists an elliptic curve over $K$ whose Mordell-Weil group contains
	$\Z/2 \times \Z/12$. \label{2,12torsion1}
\item There exist infinitely many such curves. \label{2,12torsionfinite}
\item $\rank J_1(2,12)(K) \ge 1.$ \label{2,12rank}
\end{enumerate}
\end{theorem}

\begin{theorem} \label{3,9torsion}
The following are equivalent:
\begin{enumerate}[(1)]
\item There exists an elliptic curve over $K$ whose Mordell-Weil group contains
	$\Z/3 \times \Z/9$. \label{3,9torsion1}
\item There exist infinitely many such curves. \label{3,9torsionfinite}
\item $\rank J_1(3,9)(K) \ge 1.$ \label{3,9rank}
\end{enumerate}
\end{theorem}

\begin{theorem} \label{4,8torsion}
The following are equivalent:
\begin{enumerate}[(1)]
\item There exists an elliptic curve over $K$ whose Mordell-Weil group contains
	$\Z/4 \times \Z/8$. \label{4,8torsion1}
\item There exist infinitely many such curves. \label{4,8torsionfinite}
\item $\rank J_1(4,8)(K) \ge 1.$ \label{4,8rank}
\end{enumerate}
\end{theorem}

\begin{theorem} \label{6,6torsion}
The following are equivalent:
\begin{enumerate}[(1)]
\item There exists an elliptic curve over $K$ whose Mordell-Weil group contains
	$\Z/6 \times \Z/6$. \label{6,6torsion1}
\item There exist infinitely many such curves. \label{6,6torsionfinite}
\item $\rank J_1(6,6)(K) \ge 1.$ \label{6,6rank}
\end{enumerate}
\end{theorem}

\section{Rational points obtained by $K$-curves}

In the case $X_1(M,MN)$ has genus $\ge 2$, i.e., for $X_1(N)$ for $N = 13,16,18$
(see \cite{IsMo} and \cite[Theorem 2.3.]{JeKi}. In \cite{IsMo}, there was a mistake to treat $X_\Delta$.),
even if the rank of $J_1(N)(K)$ is not $0$, we do not know the existence of an elliptic curves whose
Mordell-Weil group contains $\Z/N$.
However if the rank is small, generalizing the results of \cite[Theorem 9.]{BBDN},
we obtain a property about such elliptic curves.
Before we state the theorem, note that for $N = 13,16,18$,
there exists only one normalized eigenform of level $\Gamma_1(N)$ up to Galois conjugate
(for example, use Magma.)
and the number field associated to it is quadratic.

\begin{theorem} \label{falseCM}
Let $N = 13,16$ or $18$,
and $F$ be the number field associated to a newform of level $\Gamma_1(N)$.
Let $K$ be a multi-quadratic number field and assume that $\dim_F J(K) \otimes \Q \le 2$ or $[K : \Q] \le 2$.
Then for every elliptic curve $E$ over $K$ which has a torsion point $P$ of order $N$,
there exists a subfield $M$ of $K$ of index $2$ (i.e., $[K:M] = 2$) such that
$E$ is an $M$-curve, i.e., $E$ and its Galois conjugate with respect to $K/M$ is isogenous over $K$.
Moreover we have the following:
\begin{itemize}
\item The case $N = 13$:
    \begin{enumerate}[a.]
    \item The elliptic curve $E$ has false CM by $\Q(\sqrt{-1})$, in particular the rank of $E(K)$ is even.
    \item The elliptic curve $E$ and its Galois conjugate over $M$ is isomorphic over $K$.
    \item There exists a quadratic twist $E'$ of $E$ which is defined over $M$,
        and for every such a twist, $E'$ has CM by $\mathbb{Q}(\sqrt{-1})$ and has a cyclic $N$ isogeny over $M$.
    \end{enumerate}
\item The case $N = 16$:
    \begin{enumerate}[a.]
    \item The elliptic curve $E$ is defined over $M$.
    \item[c.] The cyclic group of order $N$ generated by $P$ is defined over $M$.
    \end{enumerate}
\item The case $N = 18$:
    \begin{enumerate}[a.]
    \item The elliptic curve $E$ has false CM by $\Q(\sqrt{-2})$, in particular the rank of $E(K)$ is even.
    \item There exists an isogeny over $K$ of degree $2$ between $E$ and its Galois conjugate over $M$.
    \end{enumerate}
\end{itemize}
\end{theorem}

For the definition of false CM, see \cite{BBDN}.

We show it by computing divisors of the modular curve $X$ we consider
which generates the torsion part of the Mordell-Weil group of $J(K)$.

First we recall the following fundamental lemma.

\begin{lemma} \label{theline}
Let $k$ be a field of characteristic $\ne 2$,
$X$ a hyperelliptic curve of genus $g \ge 2$ over $k$,
$J$ its Jacobian,
$x : X \to \PP^1$ a morphism of degree $2$,
and $A$ be the divisor of degree $2$ of $X$ which is the pullback of a point under $x$.
Let $\varphi : X^{(2)} \to J$ be the map $D \mapsto [D-A]$, where $X^{(2)}$ is the 2nd symmetric product of $X$.
Then the fiber of $\varphi$ at $0 \in J(k)$ is isomorphic to $\PP^1$ (we call it "the line"),
and $\varphi$ gives an open immersion from $X^{(2)} - \text{(the line)}$ to $J - 0$.
(If $g = 2$, then this map is in actually an isomorphism, but we will not need it.)
\end{lemma}

\begin{proof}
Since the degree $2$ map $X \to \PP^1$ is induced by the morphism defined by the canonical divisor,
$X$ has only one degree $2$ map to $\PP^1$ up to automorphisms of $X$ and of $\PP^1$.
Since the $d$-th symmetric product is the Hilbert scheme which classifies effective relative Cartier divisors of degree $d$,
we have the result.
\end{proof}

We quote the crucial proposition from \cite[section 4.4, 4.5, 4.6]{BBDN}:

\begin{proposition}[\cite{BBDN}] \label{bbdn}
Let $N = 13,16$ or $18$.
Let $K/M$ be a quadratic extension of number fields with the Galois group $\left< \sigma \right>$,
$P \in Y_1(N)(K)$ a point associated to an elliptic curve $E$,
and let $D := P + \sigma P \in X^{(2)}(M)$.
If $D$ is in "the line", then $E$ and $\sigma E$ are isogenous over $K$,
and moreover we have the same results as in the Theorem \ref{falseCM}.
\end{proposition}

By this proposition, what we need to show is that there are no "exceptional" elliptic curves,
i.e., we need to show that every elliptic curve gives a divisor in "the line".
We check it by computing the Mordell-Weil group explicitly using Magma:

\begin{lemma} \label{exceptional}
Let $N = 13,16$ or $18$,
$K$ the maximal multi-quadratic number field contained in $\Q(\zeta_N)$.
Fix a model of $X_1(N)$ mentioned in this paper.
Then for every $D \in X_1(N)^{(2)}(K)$ which is not in "the line",
we have that $D$ is either $P_1 + P_2$ for $P_i \in X_1(N)(K)$ or
$P + \sigma P$ for a quadratic point $P$ of $X_1(N)_K$, where $\sigma$ is the non-trivial element
of the Galois group of $K(P)/K$.
And in the later case, we have that the field extension $K(P)/\Q$ is not multi-quadratic.
\end{lemma}

\begin{proof}
Note that we know that $Y(K) = \varnothing$ and know $\# X(K)$.
(By Theorem \ref{Jacobian13}, \ref{Jacobian16} and \ref{Jacobian18},
we have that the rank of $J(K)$ is zero, so by Theorem \ref{13torsion}, \ref{16torsion} and \ref{18torsion},
there are no non-cuspidal points in $X(K)$.)
Since we know a model of $X$ and the equation whose roots are precisely the $x$-coordinates of the cusps,
we can compute the number of elements of $(X^{(2)} - \text{the line})(K)$ generated by the cusps,
i.e., the number of divisors of the form of $P_1 + P_2$ for $P_i \in X_1(N)(K)$ satisfying $x(P_1) \ne x(P_2)$.
It is $18$ for $N = 13$, $47$ for $N = 16$ and $50$ for $N = 18$.
On the other hand, by Lemma \ref{theline} and by Theorem \ref{Jacobian13}, \ref{Jacobian16} and \ref{Jacobian18},
we know that the number of elements of $(X^{(2)} - \text{the line})(K)$ is less than or equal to
$18$ for $N = 13$, $79$ for $N = 16$ and $62$ for $N = 18$.

According to Magma, we know the Mumford representation of every element of $J(K)$,
i.e., we know pairs $(a(T), b(T))$, where $a,b$ are polynomial over $K$,
which give all elements of $J(K).$
For such a pair $(a,b)$, if $a$ is of degree $2$ and is irreducible,
then for its root $\alpha$, we have that $P = (\alpha, b(\alpha))$ is a quadratic point of $X_K$,
and moreover $x(P) \ne x(\sigma P)$, i.e., the divisor $P + \sigma P$ is not in "the line",
where $\sigma$ is the non-trivial conjugate of $K(P)/K$.
Using Magma, we obtain that the number of such pairs $(a,b)$ is $32 = 79 - 47$ if $N = 16$ and $12 = 62 - 50$ if $N = 18$.
Finally for such a pair, we have that the splitting field of $a$ over $K$ is not a
multi-quadratic number field over $\Q$.
This gives the result.
\end{proof}

\begin{proof}[Proof of Theorem \ref{falseCM}]
Assume that there exists $P \in Y(K)$ such that for every subfield $M$ of $K$ of index $2$,
we have that $P + \sigma P \in X^{(2)}(M)$ is not in "the line",
where $\sigma$ is the non-trivial conjugate of $K/M$.
Then we have, for every such a field $M$, that the dimension of $J(M) \otimes \Q$ over $F$ is $\ge 1$:
Suppose the contrary, i.e., the rank of $J(M)$ is zero for every such a field $M$.
Then $J(M) = J(M)_\tors$, and it is $J(M \cap \Q(\zeta_N))$ by section \ref{Jacobian}.
Thus since $(X^{(2)} - \text{the line})(M)$ injects into $J(M)$,
we have that $P + \sigma P$ is in $(X^{(2)} - \text{the line})(M \cap \Q(\zeta_N))$.
Since $P$ is not cuspidal, by Lemma \ref{exceptional}, this is a contradiction,
i.e., we have that $\dim_F J(M) \otimes \Q \ge 1$.
Therefore by Proposition \ref{bbdn} and by lemma below, we have the results.
\end{proof}

\begin{lemma}
Let $k$ be a field, $K$ a multi-quadratic extension of $k$ of degree $\ge 4$
and $J$ be the Jacobian variety of a hyperelliptic curve over $k$,
which has an action of an order of a number field $F$.
Assume for every subfield $M$ of $K$ of index $2$ that $\dim_F J(M) \otimes \Q \ge 1$.
Then we have that $\dim_F J(K) \otimes \Q \ge 3$.
\end{lemma}

\begin{proof}
By Lemma \ref{exp2twists}, this is an easy proposition of linear algebra over $\F_2$:
For a vector space $V$ over $\F_2$ and for distinct nonzero points $x,y \in V$,
there exists a hyperplane $H \ni 0$ which does not contain neither $x$ nor $y$.
\end{proof}

\textbf{Acknowledgements.}
We would like to thank my supervisor Takeshi Saito for his kind and helpful advice.

\end{document}